\documentclass[12pt]{amsart}
\usepackage{amssymb,amscd,amsmath,amsthm}
 \newtheorem{thm}{Theorem}[section]
 
 \newtheorem{lem}[thm]{Lemma}
 \newtheorem{prop}[thm]{Proposition}
 \theoremstyle{definition}
 \newtheorem{defn}[thm]{Definition}
 \theoremstyle{remark}
 \newtheorem{ex}[thm]{Example}
 \theoremstyle{remark}
 \newtheorem{rem}[thm]{Remark}
 \numberwithin{equation}{section}
 \newcommand{\eps}{\varepsilon}

 \newcommand{\Real}{\mathbb{R}}

  \newcommand{\dd}[1]{\frac{d}{d #1}}
 \newcommand{\ddtwo}[1]{\frac{d^2}{d #1^2}}

 \newcommand{\ppdd}[1]{\frac{\partial^2}{\partial #1^2}}

 \newcommand{\MTW}{MTW}
 \newcommand{\cexp}{\text{c-exp}}

\newcommand{\marginnote}[1]
{%\mbox{}\marginpar{\center{\hspace{0pt}\tiny{\bf#1}}}
}

\newcounter{jy}

\begin{document}
\title[New examples satisfying Ma-Trudinger-Wang conditions]{New examples on spaces of negative sectional curvature satisfying Ma-Trudinger-Wang conditions}

\author{Paul W.Y. Lee}
\email{plee@math.berkeley.edu}
\address{Department of Mathematics, University of California at Berkeley, 970 Evans Hall \#3840 Berkeley, CA 94720-3840 USA}

\author{Jiayong Li}
\email{jiayong.li@utoronto.ca}
\address{Department of Mathematics, University of Toronto, Room 6290, 40 St. George Street, Toronto, Ontario, Canada M5S 2E4}

\date{\today}
\thanks{The first author was supported by the NSERC postdoctoral fellowship.}

\begin{abstract}
In this paper, we study the Ma-Trudinger-Wang (MTW) conditions for cost functions $c$ which are of the form $c=l\circ d$, where $d$ is a Riemannian distance function with constant sectional curvature. In this case, the MTW conditions are equivalent to some computable conditions on the function $l$. As a corollary, we give some new costs on Riemannian manifolds of constant negative curvature for which the MTW conditions are satisfied. 
\end{abstract}

\maketitle

\section{Introduction}

The problem of finding the most efficient strategy to transport one mass to another is called the problem of optimal transportation. More precisely, let $\mu$ and $\nu$ be two Borel probability measures on a manifold $M$ and let $c:M\times M\to\Real$ be a cost function. Let $\varphi:M\to M$ be a map which pushes $\mu$ forward to $\nu$. Here the push forward of a measure $\mu$ by a Borel map $\varphi$ is the measure defined by $\varphi_*\mu(U)=\mu(\varphi^{-1}(U))$ for all Borel sets $U$ contained in $M$. The total cost of this transport strategy is given by 
\[
\int_Mc(x,\varphi(x))d\mu(x). 
\]

The map which minimizes the above total cost is called the optimal map and this minimization problem is the optimal transportation problem (see \cite{Br,Mc,BeBu,AgLe,FiRi1} for various results on existence and uniqueness of optimal maps). Recently, there have been a series of breakthroughs in understanding regularity of this optimal map in a series of papers \cite{MaTrWa,TrWa1,TrWa2,Lo1,Lo2,KiMc1}. The key to the whole regularity theory lies in certain conditions introduced by Ma, Trudinger, and Wang, called the Ma-Trudinger-Wang (MTW) conditions \cite{MaTrWa} (see section 2 for the definitions). Very little is known about this condition and there are very few known examples which satisfy it (\cite{Lo1,KiMc2,FiRi2,LeMc}). 

In this paper, we consider cost functions $c$ which are composition of a function $l$ with a Riemannian distance function $d$ of constant sectional curvature. More precisely, $c=l\circ d$. The main theorems (Theorem \ref{2ineq} and \ref{3ineq}) give conditions on the function $l$ which are both necessary and sufficient for the corresponding cost $c=l\circ d$ to satisfy the MTW conditions. Moreover, these conditions on the function $l$ are computable, which is not the case for the MTW conditions in general. As a result, we find new examples on manifolds of constant sectional curvature $-1$ which satisfy the MTW conditions. More precisely, we have 

\begin{thm}\label{eg1}
Let $d$ be the Riemannian distance function on a manifold of constant sectional curvature $-1$, then the cost functions 
\[
-\cosh \circ d \text{ and } -\log \circ (1+\cosh) \circ d
\] 
satisfy the strong MTW condition and the cost functions 
\[
\pm \log \circ \cosh \circ d
\] 
satisfy the weak MTW condition.
\end{thm}

\begin{rem}
For the cost $-\cosh\circ d$ on the hyperbolic space, an alternative approach to the understanding of the MTW condition can be found in \cite{Li}. It is based on the Minkowski space hyperboloid model of the hyperbolic space.
\end{rem}

On manifolds of constant sectional curvature $1$, we also find the following new example. 

\begin{thm}\label{eg2}
Let $d$ be the Riemannian distance function on a manifold of constant sectional curvature $1$, then the cost function 
\[
-\log \circ (1+\cos) \circ d
\] 
satisfies the strong MTW condition. 
\end{thm}

It is known that the square of the Euclidean distance $|x-y|^2$ satisfies the weak MTW condition. In the final section, we give sufficient conditions for a perturbation of the form $l_\epsilon(|x-y|)$ with $l_0(z)=z^2$ to satisfy the strong MTW condition (Theorem \ref{perturb}). As a corollary, we obtain the following. 

\begin{thm}\label{eg3}
If we fix a positive constant $b$ and let 
\[
l_\eps(z) = z^2/2 -\eps z^4,
\] 
then the costs 
\[
l_\eps(|x-y|)
\] 
satisfy the strong MTW condition on the set $\{ (x,y) \in \mathbb{R}^{2n}: |x-y| \leq b \}$ for all sufficiently small $\eps>0$.
\end{thm}

\

\section{Background: The Ma-Trudinger-Wang curvature}

In this section we will review some basic results from the theory of optimal transportation needed in this paper. The theorems stated in this section hold true with more relaxed assumptions and they can be found, for instance, in \cite{Vi}.

Let $\mu_1$ and $\mu_2$ be two Borel probability measures of the manifold $M$. Assume that the support of the measures $\mu$ and $\nu$ are contained in the open subsets $\mathcal M_1$ and $\mathcal M_2$, respectively, of $M$. The optimal transportation problem corresponding to the cost function $c:M\times M\to\Real$ is the following minimization problem:

Minimize the functional
\[
 \int_Mc(x,\varphi(x))d\mu(x)
\]
among all Borel maps $\varphi:M\to M$ which push forward the measure $\mu$ to the other measure $\nu$ (i.e. $\nu(\varphi^{-1}(U))=\mu(U)$ for Borel subsets $U$ in the manifold $M$). 

Minimizers of the above optimal transportation problem are called optimal maps. To study existence, uniqueness, and regularity of optimal maps, we need the following basic assumptions on the cost $c$ and the sets $\mathcal M_1$ and $\mathcal M_2$. 
\begin{itemize}
\item \textbf{(A0) Smoothness:} the cost $c$ is $C^4$ smooth on the product $\mathcal M_1\times\mathcal M_2$, 
\item \textbf{(A1) Twist condition:} for each fixed $x$ in the set $\mathcal M_1$, the map $y \longmapsto -\partial_x c(x,y)$ from the set $\mathcal M_2$ to the cotangent space $T^*_x M$ at $x$ is injective, 
\end{itemize}

The above two conditions are motivated by the following result, which can be found in \cite{Vi}. 
 
\begin{thm}\label{optimal}(Existence and uniqueness of optimal maps)
Assume that the measure $\mu$ is absolutely continuous with respect to the Lebesgue measure and the cost $c$ satisfies the assumptions \textbf{(A0)} and \textbf{(A1)}. Then there is a Lipschitz function $f$ such that the map 
\[
z\mapsto (-\partial_xc)^{-1}(df_z)
\]
is a solution to the above optimal transportation problem. Moreover, it is unique $\mu$-almost everywhere.
\end{thm}

The map $\alpha\mapsto (-\partial_xc)^{-1}(\alpha)$ appeared in Theorem \ref{optimal} is called the cost exponential map. More precisely, let $\mathcal V_x$ be the subset of the cotangent space $T_x^* M$ defined by
\[
\mathcal V_x = \{-\partial_xc(x,y)\in T_x^* M|\ y \in \mathcal M_2\}
\]
and let $\mathcal V$ be the corresponding bundle defined by $\mathcal V =\bigcup_{x\in \mathcal M_1}\mathcal V_x$.

\begin{defn}[Cost exponential map]
The cost exponential map $\cexp_x:\mathcal V_x \rightarrow M$ is defined to be the inverse of the map $y  \mapsto -\partial_x c(x, y)$ (i.e. $\text{c-exp}_x(\alpha) = y \text{ if and only if } \alpha = -\partial_x c(x,y)$). We will denote the cost exponential map by $\cexp$ if we consider it as a map defined on the bundle $\mathcal V$. 
\end{defn}

\begin{ex} \label{dsquared} 
Let $M$ be a complete Riemannian manifold with Riemannian metric $\left<\cdot,\cdot\right>$. Let $x$ be a point on the manifold $M$ and let $\exp_x$ be the exponential map restricted to the tangent space $T_xM$. Let $U_x$ be the largest open subset of $T_xM$ containing the origin on which the restriction of the exponential map $\exp_x$ is a diffeomorphism onto its image. The cut locus $\text{cut}(x)$ at the point $x$ is the complement of the image $\exp_x(U_x)$. We will denote the union of all the cut locus in $M$ by $\text{cut}(M)$. More precisely, it is a subset of the product manifold $M\times M$ defined by $\text{cut}(M)=\bigcup_{x\in M}\{x\}\times \text{cut}(x)$. 

Let $d$ be the Riemannian distance function. It is known that the cost function $c=d^2/2$ is smooth outside the cut locus $\text{cut}(M)$. Moreover, if we identify the tangent bundle with the cotangent bundle by the Riemannian metric $\left<\cdot,\cdot\right>$, then the exponential map and the cost exponential map coincide (see Lemma \ref{cexp} for a proof). Therefore, the cost $d^2/2$ satisfies conditions \textbf{(A0)} and \textbf{(A1)} on any set $\mathcal M_1\times \mathcal M_2$  which is outside the cut locus $\text{cut}(M)$ (i.e. $\mathcal M_1\times \mathcal M_2\subseteq M\times M\setminus \text{cut}(M))$. 
\end{ex}

For the regularity theory of optimal maps, we also need the cost exponential map $\cexp$ to be smooth. More precisely, 
\begin{itemize}
\item \textbf{(A2) Non-degeneracy:}
the map $p \mapsto -\partial_y\partial_xc(x,y)(p)$ from the tangent space $T_y M$ to the cotangent space $T_x^* M$ is bijective for all pairs of points $(x,y)$ in the set $\mathcal M_1\times\mathcal M_2$.
\end{itemize}

Next, we introduce the most important object in the regularity theory of optimal maps, called the Ma-Trudinger-Wang (MTW) curvature.

\begin{defn}[The Ma-Trudinger-Wang curvature] \label{MTWdefn}
The MTW curvature $\MTW :TM\oplus\mathcal V\oplus T^*M\to\Real$ is defined by 
\[
\MTW_x(u,\alpha,\beta):=-\frac{3}{2}\partial^2_s\partial^2_t\Big|_{s=t=0} c(\gamma(t),\cexp_x(\alpha+s\beta))
\]
where $\gamma(\cdot)$ is any curve with initial velocity $\dot\gamma(0)=u$. 
\end{defn}

Finally, the main assumptions, the MTW conditions, are defined using the MTW curvature as follows: 

\begin{itemize}
\item \textbf{(A3w) Weak MTW condition:} the cost $c$ satisfies the weak MTW condition on $\mathcal M_1\times\mathcal M_2$ if $\MTW_x(u,\alpha,\beta)\geq 0$ whenever $\alpha$ is contained in $\mathcal V_x$ and $\beta(u)=0$. 
\item \textbf{(A3s) Strong MTW condition:} the cost $c$ satisfies the strong MTW condition on $\mathcal M_1\times\mathcal M_2$ if it satisfies the weak MTW condition and $\MTW_x(u,\alpha,\beta)=0$ only if  $u=0$ or $\beta=0$. 
\end{itemize}
The relevance of the MTW conditions to the regularity theory of optimal maps can be found in \cite{MaTrWa,TrWa1,TrWa2,Lo1,Lo2,KiMc1,FiLo,LoVi,FiKiMc}. Other variants of the MTW conditions which are related to the regularity theory of optimal maps can be found in \cite{FiRi2,FiRiVi1,FiRiVi2}.  

Next, we consider cost functions $c$ on Riemannian manifolds which are composition of the Riemannian distance function $d$  by a smooth function $l$. More precisely, $c(x,y) = l(d(x,y))$. We end this section by stating the following theorem for which the proof will be given in the appendix. It provides simple conditions on the function $l$ which guarantee the conditions \textbf{(A0)} - \textbf{(A2)} are satisfied by the cost $c=l\circ d$. For the convenience of notations, we will consider $l$ as a function defined on the whole real line $\mathbb R$. Note that, in the theorem, we identify the tangent and the cotangent bundle of the manifold $M$ using the given Riemannian metric. This identification will be applied thought out this paper without mentioning.

\begin{prop} \label{a1a2}
Assume that the function $l$ is a smooth even function for which the second derivative is either positive or negative (i.e. either $l''>0$ or $l''<0$), then the cost $c = l\circ d$ satisfies conditions \textbf{(A0)} - \textbf{(A2)} on each subset $\mathcal M_1\times \mathcal M_2$ outside the cut locus $\text{cut}(M)$. Moreover, the cost exponential map $\cexp$, in this case, is given by 
\[
\cexp_x(v) = \exp_x\left(\frac{(l')^{-1}(|v|)}{|v|}v\right). 
\]
\end{prop}

For the rest of this paper, we will work under the assumptions of Proposition \ref{a1a2}. 

\

\section{The Ma-Trudinger-Wang curvature and the Jacobi map}

In this section we write down the the MTW curvature in terms of the Jacobi fields. To do this, let us recall the definition of the Jacobi map introduced in \cite{LeMc}. Let $x$ and $y$ be two points on the manifold $M$ which can be connected by a unique minimizing geodesic $\gamma(\cdot)$ (i.e. $\gamma(0)=x$ and $\gamma(1)=y$). Jacobi fields defined along the geodesic $\gamma$ are solutions to the following Jacobi equation: 
\begin{equation} \label{Jacobieqn}
D_{\dd \tau}D_{\dd \tau} J(\tau) + R(\gamma'(\tau),J(\tau))\gamma'(\tau) = 0,
\end{equation}
where $D$ denotes the covariant derivative. 

This second order ordinary differential equation has a unique solution if we prescribe either its boundary values $J(0)$ and $J(1)$, or its initial values $J(0)$ and $D_{\dd\tau} J(0)$. The Jacobi map is defined as the map which takes the boundary conditions and gives the initial conditions. More precisely, 

\begin{defn}
Let $J(\cdot)$ be the Jacobi field along the geodesic $\gamma(\cdot)$ defined by the conditions $J(0)=u$, $J(1)=0$, and $J(\tau)\neq 0$ for $0<\tau < 1$. The \textbf{Jacobi map} $\mathcal J$ is given by 
\[
\mathcal J(u,y) = D_{\dd\tau} J(0).
\]
\end{defn}

Let $l$ be a function which satisfies the assumptions in Theorem \ref{a1a2}. According to Definition \ref{MTWdefn} and Theorem \ref{a1a2}, the MTW curvature of  the cost function $c=l\circ d$ is given by 
\[
\MTW(u,v,w)=-\frac{3}{2}\partial^2_s\partial^2_t \,l(d(\exp(tu), \sigma(s)))\Bigg|_{s=t=0},
\]
where $\sigma(\cdot)$ is the curve defined by $\sigma(s)=\exp\left(\frac{(l')^{-1}(|v+sw|)}{|v+sw|}(v+sw)\right)$.

The connection between the MTW curvature and the Jacobi map $\mathcal J$ is given by the following theorem.
\begin{thm} \label{cross}
The MTW curvature is given in terms of the Jacobi map $\mathcal J$ by
\[
\begin{split}
\MTW(u,v,w)=
&\frac{3}{2}\ppdd s\Bigg|_{s=0} \Bigg[\frac{|v+sw|}{h(|v+sw|)}\left<u,\mathcal J(u,\sigma(s))\right>- \\
& -\frac{\left< v+sw,u\right>^2}{|v+sw|^2h'(|v+sw|)}+\frac{\left< v+sw,u\right>^2}{|v+sw|h(|v+sw|)}\Bigg], 
\end{split}
\]
where $h$ is the inverse of the function $l'$. 
\end{thm}

\begin{proof}
Let us denote the geodesic $t\mapsto\exp_x(tu)$ by $\gamma(t)$ and let $\tau \mapsto \varphi(\tau,t,s)$ be the constant speed geodesic starting from $\gamma(t)$ and ending at $\sigma(s)$ (i.e. $\varphi(0,t,s)=\gamma(t)$ and $\varphi(1,t,s)=\sigma(s)$). By Lemma \ref{cexp}, we get
\[
-\partial_t\, l(d(\gamma(t),\sigma(s))) 
= \frac{l'(d(\gamma(t),\sigma(s)))}{d(\gamma(t),\sigma(s))} \left< \dot \gamma(t), \partial_\tau\varphi\right>\Big|_{\tau=0},
\]

If we differentiate the above equation with respect to $t$ again and apply the torsion free condition of covariant derivative, then we have 
\begin{equation}\label{cross1}
\begin{split}
&-\partial^2_t\,l(d(\gamma(t),\sigma(s)))\Big|_{t=0}\\
&= H(t,s)\left< u, \partial_\tau\varphi\right> 
+ \frac{l'(d(x,\sigma(s)))}{d(x,\sigma(s))} \left< u, D_{\partial_\tau} \partial_t\varphi\right>\Big|_{t=\tau=0},
\end{split}
\end{equation}
where $H(t,s)=\partial_t \left( \frac{l'(d(\gamma(t),\sigma(s)))}{d(\gamma(t),\sigma(s))} \right)$. 

For each fixed $s$, the set of curves defined by $\tau \mapsto \varphi(\tau,t,s)$ is a family of geodesics between $\gamma(t)$ and the point $\sigma(s)$. Therefore, $\tau \mapsto \partial_t \varphi\Big|_{t=0}$ defines a Jacobi field. Moreover, this Jacobi field has boundary values $\partial_t \varphi\Big|_{t=\tau=0} = u$ and $\partial_t\varphi\Big|_{t=0,\tau=1}= 0$. Therefore, by the definition of the Jacobi map, (\ref{cross1}) becomes 
\begin{equation}\label{cross1-2}
\begin{split}
&-\partial^2_t\,l(d(\gamma(t),\sigma(s)))\Big|_{t=0}\\
&= H(t,s)\left< u, \partial_\tau\varphi\right>\Big|_{t=\tau=0} 
+ \frac{l'(d(x,\sigma(s)))}{d(x,\sigma(s))} \left< u, \mathcal J(u,\sigma(s))\right>.
\end{split}
\end{equation}

Note that $\partial_\tau\varphi\Big|_{t=\tau=0}$ is the initial velocity $\frac{h(|v+sw|)}{|v+sw|}(v+sw)$ of the geodesic between $x$ and $\sigma(s)$, it follows that $d(x,\sigma(s))=|h(|v+sw|)|$. Since we assume that $l'$ is odd (see the comment following Proposition \ref{a1a2}), (\ref{cross1-2}) becomes  

\begin{equation}\label{cross2}
\begin{split}
&-\partial^2_t\,l(d(\gamma(t),\sigma(s)))\Big|_{t=0}\\
&= H(t,s)\left< u, \partial_\tau\varphi\right>\Big|_{t=\tau=0} 
+ \frac{|v+sw|}{h(|v+sw|)} \left< u, \mathcal J(u,\sigma(s))\right>.
\end{split}
\end{equation}

If we apply Lemma \ref{cexp} again to the term involving $H$, we get 
\begin{equation}\label{cross3}
\begin{split}
&H(t,s)\left<u,\partial_\tau\varphi\right>\Big|_{t=\tau=0}\\
&= \left( -\frac{l''(d(x,\sigma(s)))}{d(x,\sigma(s))^2} + \frac{l'(d(x,\sigma(s)))}{d(x,\sigma(s))^3} \right)\cdot \left<u,\partial_\tau\varphi\right>^2\Big|_{t=\tau=0}.
\end{split}
\end{equation}

If we apply again the facts that $l'$ is odd, $\partial_\tau\varphi\Big|_{t=\tau=0}=\frac{h(|v+sw|)}{|v+sw|}(v+sw)$, and $d(x,\sigma(s)) = |h(|v+sw|)|$, then (\ref{cross3}) becomes 
\begin{equation}\label{cross4}
\begin{split}
&H(t,s)\left<u,\partial_\tau\varphi\right>\Big|_{t=\tau=0}\\
&= -\frac{l''(h(|v+sw|))\left<u,v+sw\right>^2}{|v+sw|^2} + \frac{\left<u,v+sw\right>^2}{|v+sw|h(|v+sw|)}.
\end{split}
\end{equation}

Since $h=(l')^{-1}$, we have $l''(h(s))=\frac{1}{h'(s)}$ and (\ref{cross4}) becomes 
\begin{equation}\label{cross5}
\begin{split}
&H(t,s)\left<u,\partial_\tau\varphi\right>\Big|_{t=\tau=0}\\
&= -\frac{\left<u,v+sw\right>^2}{h'(|v+sw|)|v+sw|^2} + \frac{\left<u,v+sw\right>^2}{|v+sw|h(|v+sw|)}.
\end{split}
\end{equation}

Therefore, we can combine this with (\ref{cross2}) and (\ref{cross5}). This finishes the proof of the theorem. 
\end{proof}

\

\section{The Ma-Trudinger-Wang curvature on space forms}

In this section, we assume that the manifold $M$ is a space form (i.e. a Riemannian manifold of constant sectional curvature). In this case, the Jacobi map $\mathcal J$, and hence the MTW curvature for the cost $c=l\circ d$, can be written down explicitly. To do this, let us fix a tangent vector $v$ in the tangent space $T_xM$ at the point $x$. For each vector $u$ in the same tangent space, we let $u_0$ and $u_1$ be the components of $u$ contained in the subspace spanned by $v$ and its orthogonal complement, respectively. Recall that $h$ is the inverse of the function $l'$.

\begin{thm}\label{MTWspace}
Let $A$ and $B$ be the functions defined by  
\[
A(z)=\frac{1}{h'(z)}, \quad
B(z)=
\begin{cases}
z\coth(h(z)) & \text{ if } K=-1, \\
\frac{z}{h(z)} & \text{ if } K=0, \\
z\cot(h(z)) & \text{ if } K=1.\\
\end{cases}
\]
Then the MTW curvature is given by
\begin{equation*}\label{ABprimes}
\begin{split}
&\MTW(u,v,w) =-\frac{3}{2}\Bigg[A''(|v|)|u_0|^2|w_0|^2+B''(|v|)|u_1|^2|w_0|^2 +\\
& +\frac{A'(|v|)}{|v|}(|u_0|^2|w_1|^2+4\left<u_0,w_0\right>\left<u_1,w_1\right>)+\\
&+\frac{B'(|v|)}{|v|}(|u_1|^2|w_1|^2-4\left<u_0,w_0\right>\left<u_1,w_1\right>) +\\
& +\frac{2(A(|v|)-B(|v|))}{|v|^2}(\left<u_1,w_1\right>^2 -|u_0|^2|w_1|^2-2\left<u_0,w_0\right>\left<u_1,w_1\right>)\Bigg].
\end{split}
\end{equation*}
\end{thm}

Let us begin the proof by writing down the formula for the Jacobi map in a space form. 

\begin{lem} \label{Jacobimap}
The Jacobi map $\mathcal J$ on a space form of sectional curvature $K$ is given by 
\[
\mathcal J(u,\exp(v)) = 
\begin{cases}
-u_0-|v|\coth(|v|)u_1 & \text{ if } K=-1, \\
-u & \text{ if } K=0, \\
-u_0-|v|\cot(|v|)u_1 & \text{ if } K=1. \\
\end{cases}
\]
\end{lem}

\begin{proof} [Proof of Lemma \ref{Jacobimap}]
We will only give the proof for the case $K=-1$. The proofs for the other two cases are similar and will be omitted. 
Let $u_0(\tau)$ and $u_1(\tau)$ be the parallel translation of the vectors $u_0$ and $u_1$, respectively, along the geodesic $\tau\mapsto\exp(\tau v)$. Let $\tau\mapsto J(\tau)$ be the Jacobi field which satisfies $J(0)=u$, $J(1)=0$, and $J(\tau)\neq 0$. Recall that the Jacobi map $\mathcal J$ is given by 
\[
\mathcal J(u,\exp(v))=D_{\dd \tau}J\Big|_{\tau=0}. 
\]

Let $\mathcal J_0$ and $\mathcal J_1$ be the components of $\mathcal J(u,\exp(v))$ contained in the subspace spanned by $v$ and its orthogonal complement, respectively. Let $\mathcal J_0(\tau)$ and $\mathcal J_1(\tau)$ be the parallel translation of $\mathcal J_0$ and $\mathcal J_1$, respectively, along the geodesic $\tau\mapsto\exp(\tau v)$. We claim that the solution to the Jacobi equation (\ref{Jacobieqn}) with the initial values $J(0) = u$ and $D_{\dd\tau} J(0) = \mathcal J$ is given by 
\begin{equation}\label{Jacobimap1}
J(\tau) = 
u_0(\tau) + \tau\mathcal J_0(\tau) + \cosh(|v|\tau)u_1(\tau) + \frac{\sinh(|v|\tau)}{|v|}\mathcal J_1(\tau).
\end{equation}

The above equation clearly satisfies the initial conditions $J(0) = u$ and $D_{\dd\tau} J(0) = \mathcal J$. It remains to show that $J(\cdot)$ satisfies the Jacobi equation. Since the sectional curvature of the manifold is $-1$, the Riemann curvature $R$ satisfies 
\[
R(w_1,w_2)w_1=-|w_1|^2w_2
\]
for each tangent vector $w_1$ which is orthogonal to $w_2$. 

Therefore, if we denote the geodesic $\exp(\tau v)$ by $\gamma(\tau)$, then it follows that  
\[
D_{\dd\tau}D_{\dd\tau} J=|v|^2\left(\cosh(|v|\tau)u_1(\tau) + \frac{\sinh(|v|\tau)}{|v|}\mathcal J_1(\tau)\right)=-R(\dot\gamma,J)\dot\gamma. 
\]
This finishes the proof of the claim. 

Since $J(1)=0$, it follows from (\ref{Jacobimap1}) that
\[ 
\mathcal J_0(1)=-u_0(1),\quad   \mathcal J_1(1)=-|v|\coth(|v|)u_1(1).
\] 
Finally, since parallel translations are linear isomorphism, we have 
\[
\mathcal J=\mathcal J_0(0)+\mathcal J_1(0)=-u_0-|v|\coth(|v|)u_1. 
\]
\end{proof}

\begin{proof}[Proof of Theorem \ref{MTWspace}]
Let $u_0(s)$ and $u_1(s)$ be the components of $u$ contained in the subspace spanned by $v+sw$ and its orthogonal complement, respectively. It follows from Lemma \ref{Jacobimap} that
\[
\mathcal J(u,\sigma(s))=-u_0(s)-\frac{h(|v+sw|)}{|v+sw|}B(|v+sw|)u_1(s). 
\]

If we combine this with Theorem \ref{cross}, we get
\begin{equation}\label{AB}
\begin{split}
&\MTW(u,v,w)\\
&=-\frac{3}{2}\ppdd s \Bigg|_{s=0} \Bigg[A(|v+sw|)|u_0(s)|^2+B(|v+sw|)|u_1(s)|^2\Bigg].
\end{split}
\end{equation}

Since the vector $u$ is decomposed into two orthogonal components $u_0(s)$ and $u_1(s)$, we have $|u|^2=|u_0(s)|^2+|u_1(s)|^2$. After a long computation, (\ref{AB}) becomes
\begin{equation}\label{ABprimes}
\begin{split}
\MTW(u,v,w)
&=-\frac{3}{2}\Bigg[(A''(|v|)|u_0|^2+B''(|v|)|u_1|^2)V_1^2+\\
&+(A'(|v|)|u_0|^2+B'(|v|)|u_1|^2)V_2+\\
&+2(A'(|v|)-B'(|v|))\dd s|u_0(s)|^2\Bigg|_{s=0}V_1+\\
&+(A(|v|)-B(|v|))\ddtwo s|u_0(s)|^2\Bigg|_{s=0}\Bigg],
\end{split}
\end{equation}
where $V_1=\dd s|v+sw|\Bigg|_{s=0}$ and $V_2=\ddtwo s|v+sw|\Bigg|_{s=0}$. 

Another long calculation shows that
\[
\begin{split}
 &V_1=\dd s|v+sw|\Big|_{s=0}=\frac{\left<v,w_0\right>}{|v|}, \quad \\
 & V_2=\ddtwo s|v+sw|\Big|_{s=0}=\frac{|w_1|^2}{|v|},\\
 &\dd s|u_0(s)|^2\Big|_{s=0}=\frac{2\left<u_0,v\right>\left<u_1,w_1\right>}{|v|^2},\\
 &\ddtwo s|u_0(s)|^2\Big|_{s=0}=\frac{2}{|v|^2}(\left<u_1,w_1\right>^2 -|w_1|^2|u_0|^2 -2\left<u_0,w_0\right>\left<u_1,w_1\right>).
\end{split}
\]
Finally we combine these equations with (\ref{ABprimes}) and the result follows.
\end{proof}

\

\section{The Ma-Trudinger-Wang conditions on space forms}

In this section, we continue to investigate the MTW conditions for cost functions of the form $c=l\circ d$, where $d$ is a Riemannian distance function of a space form. We give computable conditions on the function $l$ which are equivalent to the MTW conditions \textbf{(A3w)} and \textbf{(A3s)}. 

Recall that $h$ is the inverse of the function $l'$. The functions $A$ and $B$ are defined by 
\[
A(z)=\frac{1}{h'(z)}, \quad B(z)=
\begin{cases}
z\coth(h(z)) & \text{ if } K=-1, \\
\frac{z}{h(z)} & \text{ if } K=0, \\
z\cot(h(z)) & \text{ if } K=1.\\
\end{cases}
\]

\begin{prop}\label{MTWsimple}
Assume that the tangent vectors $u$ and $w$ satisfy the orthogonality condition $\left<u,w\right>=0$, then the MTW curvature is given by 
\begin{align*}
\MTW(u,v,w) = -&\frac{3}{2}\Big[\alpha(|v|)|u_0|^2|w_0|^2 + \beta(|v|)|u_0|^2|w_1|^2+ \\
&+ \gamma(|v|)|u_1|^2|w_0|^2 + \delta(|v|)|u_1|^2|w_1|^2\Big],
\end{align*}
where $\alpha$, $\beta$, $\gamma$, and $\delta$ are functions defined by 
\begin{align*}
&\alpha(z) = \frac{z^2A''(z)+6(A(z)-B(z))-4z(A'(z)-B'(z))}{z^2},\\
&\beta(z) = \frac{zA'(z)-2(A(z)-B(z))}{z^2},\\ 
&\gamma(z) = B''(z), \\ 
&\delta(z) = \frac{B'(z)}{z}.
\end{align*}
\end{prop}

\begin{proof}
Since $u$ and $w$ are orthogonal, we have $\left<u_0,w_0\right>+\left<u_1,w_1\right>=0$. The result follows from this and Theorem \ref{MTWspace}.
\end{proof}

Next, we look at the conditions on the function $l$ which are equivalent to the MTW conditions. The situation in the two dimensional and the higher dimensional cases are slightly different. Let us first state the result in two dimension.

\begin{thm} \label{2ineq}
Assume that the manifold $M$ has dimension two, then the cost $c=l \circ d$ satisfies the condition \textbf{(A3w)} on any subset $\mathcal M_1\times\mathcal M_2$ outside the cut locus $\text{cut}(M)$ if and only if 
the following inequalities hold for each $z$ in the interval $[0,|l'(D)|]$, where $D$ is the diameter of the manifold $M$:
\begin{enumerate}
\item $\beta(z),\gamma(z) \leq 0$,
\item $\alpha(z)+\delta(z)\leq 2\sqrt{\beta(z)\gamma(z)}$.
\end{enumerate}
In addition, if the above non-strict inequalities are replaced by strict inequalities, then it is equivalent to the cost $c$ being \textbf{(A3s)} on any subset $\mathcal M_1\times\mathcal M_2$ outside the cut locus $\text{cut}(M)$. 
\end{thm}

\begin{proof}
By Proposition \ref{MTWsimple}, the cost $c=l \circ d$ satisfies the condition \textbf{(A3w)} on any subset $\mathcal M_1\times\mathcal M_2$ outside the cut locus $\text{cut}(M)$ if and only if 
\begin{equation}\label{ineq1}
\begin{split}
&\alpha(z)|u_0|^2|w_0|^2 + \beta(z)|u_0|^2|w_1|^2+ \\&+\gamma(z)|u_1|^2|w_0|^2 + \delta(z)|u_1|^2|w_1|^2\leq 0
\end{split}
\end{equation}
for each $z$ in the interval $[0,|l'(D)|]$ and each pair of tangent vectors $u$ and $w$ which are orthogonal $\left<u,w\right>=\left<u_0,w_0\right>+\left<u_1,w_1\right>=0$. 

First, assume that the dimension of the manifold $M$ is two. Let $\{ v_0=\frac{v}{|v|},v_1 \}$ be an orthonormal basis of the tangent space $T_xM$ and let $u=av_0+bv_1$. It follows from the orthogonality condition that the vector $w$ is of the form $w=\lambda(bv_0-av_1)$. If we substitute this into (\ref{ineq1}), then we have 
\[
\beta(z)a^4+(\alpha(z)+\delta(z))a^2b^2 + \gamma(z)b^4 \leq 0.
\]

This inequality, in turn, is equivalent to 
\[
\beta(z)\leq 0, \gamma(z)\leq 0, \alpha(z)+\delta(z)\leq 2\sqrt{\beta(z)\gamma(z)}. 
\]

A similar proof with all inequality replaced by strict inequality shows the second statement of the theorem on the condition (\textbf{A3s}).  
\end{proof}

When the manifold has dimension higher than two, there is an additional inequality on the function $\delta$ for the MTW condition. More precisely, 

\begin{thm} \label{3ineq}
Assume that the manifold $M$ has dimension greater than two, then the cost $c=l \circ d$ satisfies the condition \textbf{(A3w)} on any subset $\mathcal M_1\times\mathcal M_2$ outside the cut locus $\text{cut}(M)$ if and only if 
the following inequalities hold for each $z$ in the interval $[0,|l'(D)|]$, where $D$ is the diameter of the manifold $M$:
\begin{enumerate}
\item $\beta(z),\gamma(z),\delta(z)\leq 0$,
\item $\alpha(z)+\delta(z)\leq 2\sqrt{\beta(z)\gamma(z)}$.
\end{enumerate}
In addition, if the above non-strict inequalities are replaced by strict inequalities, then it is equivalent to the cost $c$ being \textbf{(A3s)} on any subset $\mathcal M_1\times\mathcal M_2$ outside the cut locus $\text{cut}(M)$. 
\end{thm}

\begin{proof}
If we assume that both $u_0$ and $w_0$ are nonzero and set $u'=\frac{u_1}{|u_0|}$ and $w'=\frac{w_1}{|w_0|}$, then (\ref{ineq1}) becomes 
\begin{equation}\label{ineq2}
\alpha(z) + \beta(z)|w'|^2+ \gamma(z)|u'|^2 + \delta(z)|u'|^2|w'|^2\leq 0. 
\end{equation}
The orthogonality condition $\left<u,w\right>=0$ becomes $\left<u',w'\right>=\pm 1$. 

Assume that the dimension of the manifold is great than two. Let $u''$ be a vector contained in the subspace spanned by $u'$ and $w'$ which satisfies $\left<u',u''\right>=0$ and $|u'|=|u''|$. By the orthogonality condition $\left<u',w'\right>=\pm 1$, we can let $w'=\pm \frac{u'}{|u'|^2}+bu''$. The length $|w'|$ of $w'$ is given by $|w'|^2=b^2|u'|^2+\frac{1}{|u'|^2}$. If we substitute this back into (\ref{ineq2}), then we have 
\[
\left(\beta(z)|u'|^2+\delta(z)|u'|^4\right)b^2+
\alpha(z) + \delta(z) + \frac{\beta(z)}{|u'|^2}+ \gamma(z)|u'|^2\leq 0. 
\]

The above inequality holds for all $b$ if and only if 
\[
\delta(z)|u'|^2+\beta(z)\leq 0,\quad \gamma(z)|u'|^4+(\alpha(z) + \delta(z))|u'|^2 + \beta(z)\leq 0. 
\]

This, in turn, holds for all $u'$ if and only if 
\[
\delta(z)\leq 0, \beta(z)\leq 0, \gamma(z)\leq 0
\]
and the quadratic equation $\gamma(z)x^2+(\alpha(z) + \delta(z))x + \beta(z)=0$ has at most one positive root. 

Finally the fact that the quadratic equation above has at most one positive root is equivalent to the condition $\alpha(z)+\delta(z)\leq 2\sqrt{\gamma(z)\beta(z)}$. 

A similar proof with all inequality replaced by strict inequality shows the second statement of the theorem on the condition (\textbf{A3s}).  
\end{proof}

\

\section{Perturbations of the Euclidean distance squared}

The cost $|x-y|^2$ given by square of the Euclidean distance has zero MTW curvature. Here we give an easily computable sufficient condition for a perturbation $l_\eps(|x-y|)$ of this cost to satisfy the condition \textbf{(A3s)}.

\begin{thm} \label{perturb}
Let $l_\eps (z)$ be a smooth family of smooth even functions which satisfy $l_0(z) = z^2/2$ and let $f$ be the function defined by 
\[
f(z)=\frac{\partial_\eps l'_\eps(z)}{z}\Bigg|_{\eps=0}.
\] 
Assume that there is a negative constant $k$ such that the following inequalities are satisfied for every $z$ in an interval $(0,b]$: 
\begin{enumerate}
\item $f''(z) < k$,
\item $\frac{z^2 f'''(z) - z f''(z) + 2 f'(z)}{z} < k$. 
\end{enumerate}
Then the costs $l_\eps(|x-y|)$ satisfy the conditions \textbf{(A0)-(A2)} and \textbf{(A3s)} on the set $\{ (x,y) \in \mathbb{R}^{2n}: |x-y| \leq b \}$ for all sufficiently small $\eps>0$. 
\end{thm}

\begin{proof}
Since $l_0''=1$, the functions $l_\eps$ has positive second derivative on the interval $[0,b]$ for all small enough $\eps\geq 0$. 
It follows from Proposition \ref{a1a2} that the cost $l_\eps(|x-y|)$ satisfies the conditions \textbf{(A0)-(A2)} on the set $\{ (x,y) \in \mathbb{R}^{2n}: |x-y| \leq b \}$ for all small enough $\eps$. Let $A_\eps, B_\eps, \alpha_\eps, \beta_\eps, \gamma_\eps, \delta_\eps, h_\eps$ be the functions $A, B, \alpha, \beta, \gamma, \delta, h$, respectively, defined in section 5 with the function $l$ replaced by $l_\eps$. Note that $\alpha_0(z) = \beta_0(z) = \gamma_0(z) = \delta_0(z) = 0$. Therefore, if we can show that the quantities $\partial_\eps\alpha_\eps\Big|_{\eps=0}, \partial_\eps\beta_\eps\Big|_{\eps=0}, \partial_\eps\gamma_\eps\Big|_{\eps=0}, \partial_\eps\delta_\eps\Big|_{\eps=0}$ are all negative on the interval $[0,b]$, then we can apply Theorem \ref{3ineq} and conclude that the costs $l_\eps(|x-y|)$ satisfy the strong MTW condition on the set $\{ (x,y) \in \mathbb{R}^{2n}: |x-y| \leq b \}$ for all sufficiently small $\eps>0$. 

To do this, let us first compute $\partial_\eps B'_\eps(z)\Big|_{\eps=0}$ and $\partial_\eps B''_\eps(z)\Big|_{\eps=0}$. If we differentiate the identity $l'_\eps(h_\eps(z))=z$ with respect to $\eps$, we get 
\begin{equation}\label{perturb1}
\partial_\eps h_\eps\Big|_{\eps=0}= -\partial_\eps l'_\eps\Big|_{\eps=0}.
\end{equation}

It follows from this and the definition of the function $B_\eps$ that 
\begin{equation}\label{perturb2}
\partial_\eps B_\eps\Big|_{\eps=0}  = f. 
\end{equation}

Therefore, we have $\partial_\eps B_\eps'\Big|_{\eps=0}= f'$ and $\partial_\eps B_\eps''\Big|_{\eps=0} = f''$. It follows from this and $f''<0$ that $\partial_\eps\gamma_\eps\Big|_{\eps=0}$ is negative. Note that since $f''<k<0$, it also follows that $\partial_\eps\delta_\eps(z)\Big|_{\eps=0}=\frac{f'(z)}{z}<0$. 

By (\ref{perturb1}) and the definition of the function $A_\eps$, we have 
\begin{equation}\label{perturb3}
\partial_\eps A_\eps(z)\Big|_{\eps=0}= \partial_\eps l''_0(z)=zf'(z)+f(z). 
\end{equation}

If we apply (\ref{perturb2}) and (\ref{perturb3}) to the definition of the function $\alpha_\eps$ and $\beta_\eps$, then a calculation yields 
\[
\partial_\eps \alpha_\eps(z)\Big|_{\eps=0}=\frac{z^3 f'''(z) - z^2 f''(z) + 2z f'(z)}{z^2}<k<0
\]
and
\[
\partial_\eps \beta_\eps(z)\Big|_{\eps=0} = f''(z) < k<0. 
\]  
This finishes the proof of the theorem. 
\end{proof}

\

\section{Appendix}

In this section, we give the proof of Proposition \ref{a1a2}. Let us begin by the following known result. A proof is given for completeness. 

\begin{lem}\label{cexp}
Let $x$, $y$ be a pair of points which can be connected by a unique minimizing geodesic and let $f:M\to\Real$ be the function $f(z)=\frac{1}{2}d^2(z,y)$. Then the initial velocity of the minimizing geodesic starting from $x$ and ending at $y$ is given by $-\nabla f(x)$. In other words, we have 
\[
\exp(-\nabla f(x))=y.
\]
\end{lem}

\begin{proof}
Let $t\mapsto\gamma(t)$ be the unique minimizing geodesic satisfying $\gamma(0)=x$ and $\gamma(1)=y$. We need to show that $\dd t\gamma(0)=-\nabla f(x)$. To do this, let $\varphi(t,s)$ be a variation which satisfies the condition $\varphi(t,0)=\gamma(t)$ and $\varphi(1,s)=y$. If we set $u=\dd s\varphi\Big|_{t=s=0}$, then we have 
\[
df(u)=\dd s\frac{1}{2}d^2(\varphi(0,s),y)\Big|_{s=0}=\dd s\frac{1}{2}\int_0^1\left|\dd t\varphi\right|^2dt\Big|_{s=0}. 
\]

Since $t\mapsto \varphi(t,0)=\gamma(t)$ is a geodesic, the above equation becomes 
\[
df(u)=\int_0^1\left<D_{\dd s}\dd t\varphi,\dd t\varphi\right>dt\Big|_{s=0} =\int_0^1\dd t\left<\dd s\varphi,\dd t\varphi\right>dt\Big|_{s=0}. 
\]

Since $\varphi(1,s)=y$ is independent of the variable $s$, we get the following 
\[
df(u) =-\left<\dd t\gamma\Big|_{t=0},u\right>. 
\]
Since the above equation holds for all tangent vector $u$ in the tangent space $T_xM$, the result follows. 
\end{proof}

\begin{proof}[Proof of Proposition \ref{a1a2}]

Since $d^2$ is smooth outside the cut locus (i.e. smooth on $M \times M \setminus \text{cut}(M)$) and the function $s\mapsto l(\sqrt s)$ is smooth, the cost $c=l\circ d$ satisfies condition $\textbf{(A0)}$. 

Let $(x,y)$ be a pair of points which are not contained in the cut locus $\text{cut}(M)$. Then there is a unique minimizing geodesic $\gamma(\cdot)$ which starts at the point $x$ and ends at the point $y$. Let $u=\dot\gamma(0)$ be the initial velocity of this geodesic. In other words, we have $\exp_x(u)=y$. Let $U_x$ be the open subset of the tangent space $T_xM$ on which the exponential map $\exp_x$ is a diffeomorphism onto its image. Let us denote this inverse by $\exp_x^{-1}$. Since $(x,y)$ is not contained in the cut locus, we also have $u=(\exp_x)^{-1}y$. By Lemma \ref{cexp}, we have 
\begin{equation}\label{a1a21}
-\partial_xc(x,\exp(u))=-\partial_x(l\circ d)(x,\exp(u))=\frac{l'(|u|)}{|u|}u. 
\end{equation}

By the assumption on the second derivative of the function $l$, we can define the inverse of the derivative $l'$ as $h$. If we substitute $u=\frac{h(|v|)}{|v|}v$ into (\ref{a1a21}), then we have 
\[
-\partial_xc\left(x,\exp\left(\frac{h(|v|)}{|v|}v\right)\right)=v. 
\]
It follows that $y\mapsto -\partial_xc(x,y)$ is injective and the cost $c=l\circ d$ satisfies the condition \textbf{(A1)}. 

Therefore, the $c$-exponential map $\cexp$ is given by 
\[
\cexp(v)=\exp\left(\frac{h(|v|)}{|v|}v\right). 
\]
Since the function $h$ is a smooth odd function, the $c$-exponential map is smooth. Therefore, the cost satisfies the condition \textbf{(A2)}.
\end{proof}

\section*{Acknowledgment}
We thank Professor Robert McCann for his interest in this work and various fruitful discussions with us.

\end{document}